\documentclass[leqno]{article}

\RequirePackage{amsthm,amsmath,amsfonts,amssymb}

\RequirePackage[numbers]{natbib}
\RequirePackage{graphicx}

\RequirePackage{hyperref}
\RequirePackage{breakurl}

\newcommand{\fnt}[1]{{\fontfamily{cmss}\selectfont #1}}

\newtheorem{thm}{Theorem}

\begin{document}

\title{\bf Triangle Sides for Congruent Numbers less than 10,000}
\author{David Goldberg \\
  dg.paloalto@gmail.com}
\maketitle

\begin{abstract}
  We have computed a table of the triangle sides of all congruent numbers less than 10,000, which improves
  and extends the existing public table.  We give some background on properties of the triangle sides,
  and explain how we computed our table, which is available on
  \url{https://github.com/dgpaloalto/Congruent-Numbers}
\end{abstract}



\section{Introduction}
A congruent number is a positive integer that is the area of a right triangle with rational sides.
For example 5 is congruent, being the area of the triangle with sides $3/2$, $20/3$ and $41/6$.
A basic reference is the textbook \cite{koblitz}.   We know of only a single table
giving the rational sides of congruent numbers, appearing at \cite{congruum}
and \cite{numeri}.  One is likely a copy of the other, since they share
the same glitch at $N=559$, as will be explained later.%
\footnote{Also on \cite{congruum}, $n=199$  has $Q =
 27368486201$ instead of the correct value 2736848\textbf{7}6201.}
This table gives
the sides for all congruent numbers under 1000,
but there is no information about how the numbers were computed.
The purpose of this note is to introduce an extended table for congruent numbers
up to 10000, together with a detailed explanation of how they were computed.

When $N$ is congruent, there are infinitely many representations
by right triangles.   If the sides are $\alpha$, $\beta$ and $\gamma$, with
$\gamma$ the hypotenuse, we define the ``height'' of the representation as
$\max(\alpha_1, \alpha_2, \beta_1, \beta_2)$  where $\alpha = \alpha_1/\alpha_2$,
$\beta = \beta_1/\beta_2$ are the rational sides as the quotients of relatively
prime integers.   We strive to give the representation with the triangle of minimum height, 
although we do not claim to have always achieved that.  In particular, our new table
has representations with smaller heights then the existing table
for seven different congruent numbers. Figure 1 has the details.

\begin{table}
\begin{scriptsize}
\begin{tabular}{|r|r|r|r|r|r|r|r|r|r|r|} \hline
& \multicolumn{5}{|c|}{previous table} & \multicolumn{5}{|c|}{our table} \\ \hline
  $n$ &  $\alpha_1$ & $\alpha_2$ & $\beta_1$ & $\beta_2$ & height &
   $\alpha_1$ & $\alpha_2$ & $\beta_1$ & $\beta_2$ & height \\ \hline
219 & 1752& 55 & 55 &  4 & 1752 & 264 &  13 & 949 & 44 & 949 \\
323 & 12920 & 297 & 297 & 20 & 12920 & 1785 &  92 & 3496 & 105 & 3496 \\
330 & 60 &  1 & 11 &  1 &  60 & 24 &  1 & 55   & 2 &  55 \\
410 & 1640 & 9 & 9 & 2 & 1640 & 451 & 18 & 360 & 11& 451 \\ 
434 &  496 &  3 &  21 &   4 & 496 & 279 &  10 & 280 &  9 & 280 \\
609 &   20 &  1 & 609 &  10 & 609 &  28 &   3 & 261 &  2 &  261 \\
915 & 3660 &  11 & 11 &  2 & 3660 & 244 &  63 & 945 &  2 & 945 \\\hline
\end{tabular}
\caption{The seven numbers less than 1000 that have a lower height in our new table.}
\end{scriptsize}
\end{table}

As explained in \cite{koblitz}, assuming the Birch-Swinnerton-Dyer
conjecture, there is an easily computable test due to Tunnell
for whether a number is congruent.
We use that test when selecting numbers for representation by triangle sides.
Figure 1 shows a visualization of our new table.

\begin{figure}
\includegraphics[height=2.5in]{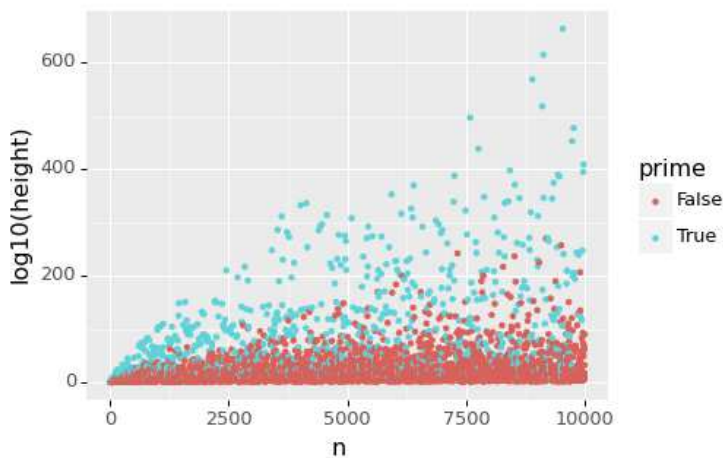}
\caption{A visualization of our table.  For each congruent number $n$, it plots
the base-10 logarithm of the height, $\max(\alpha_1, \alpha_2, \beta_1, \beta_2)$. }
\end{figure}

The table and the code that generated it
is available on GitHub as \url{https://github.com/dgpaloalto/Congruent-Numbers}.

\section{Properties of the Representation}

We might as well assume that $N$ is square free, since $N$ is congruent if and only if
its square free factor is congruent, and the sides of one are easily deduced from the other.

\begin{thm}  \label{lem:2}
  If $N$ is a square-free congruent number represented with perpendicular sides
  $\alpha_1/\alpha_2$ and $\beta_1/\beta_2$ both in reduced form,
  then $(\alpha_1, \beta_1) = 1$, $(\alpha_2, \beta_2) = 1$
  and $\alpha_1$ and $\beta_1$ have opposite parities.
\end{thm}
\begin{proof}
The definition of congruent number means that
\begin{equation} \label{eq:def}
     \alpha_1 \beta_1 = 2N \alpha_2\beta_2
\end{equation}
If $\alpha_2$ and $\beta_2$ have a common prime factor $p$, then
$p^2 \mid \alpha_1\beta_1$, so $p$ divides at last one of them, contradicting that the fractions
are in reduced form.  This gives
\begin{equation}
(\alpha_2, \beta_2) = 1
\end{equation}

From equation (\ref{eq:def}) and $N$ square free,
it follows that $\alpha_1$ and $\beta_1$  have no common odd prime factor.
To finish the proof, we only need to show that 
$\alpha_1$ and $\beta_1$ are of opposite parity.
First, they can't both be odd, since $2N \alpha_2 \beta_2$ is even.
And they can't both be even. If they were, then  $\alpha_2$ and $\beta_2$ are both odd.
Also, from the definition of congruent number,
\[
\frac{\alpha_1^2\beta_2^2 + \alpha_2^2\beta_1^2}{\beta_1^2\beta_2^2}
= \frac{\alpha_1^2}{\alpha_2^2} + \frac{\beta_1^2}{\beta_2^2} = \frac{\delta_1^2}{\delta_2^2} 
\]
which shows that numerator on the left is a square:
\begin{equation}\label{eq:pythag}
(\alpha_1 \beta_2)^2 + (\alpha_2 \beta_1)^2 = C^2
\end{equation}
for some integer $C$.  If both $\alpha_i = 2 \bmod 4$,
then when you divide equation (\ref{eq:pythag}) by 4, each summand on the
left hand side is congruent to 1 mod 4, so their sum
is congruent to 2 mod 4, which is not a possible residue for $C^2$.
Therefore one of the $\alpha_i$ must be divisible by 4
and their product is divisible by 8.  That means both sides
of equation (\ref{eq:def}) are
divisible by 8, forcing $N$ to be divisible by 4, contradicting that $N$
is square free.  This contradiction shows
that $\alpha_1$ and $\beta_1$ are of opposite parity.
\end{proof}

We strive to produce the representation of minimal height.  Note that
\begin{thm} \label{lem:1}
  The height $\max(\alpha_1, \alpha_2, \beta_1, \beta_2)$ is equal to 
  $\max(\alpha_1, \beta_1)$.
\end{thm}
\begin{proof}
From equation (\ref{eq:def}), the product $(\alpha_1/\alpha_2) (\beta_1 /\beta_2)$ is an integer
and the fractions are relatively prime, so $\alpha_2 \mid \beta_1$
and therefore $\alpha_2 \leq \beta_1$.  Similarly $\beta_2 \leq \alpha_1$.
\end{proof}


\section{Representation using $P$ and $Q$}

For a fixed $N$, there is a 1-1 mapping between a pair of
rational sides and a pair of relatively prime integers $(P,Q)$. The existing tables  \cite{congruum}
and \cite{numeri}  give the representation in terms of these $P$ and $Q$.

To go from $\alpha$ and $\beta$ to $P$ and $Q$, rewrite
equation (\ref{eq:pythag}) as:
\begin{eqnarray*}
  A &=& \alpha_1\beta_2 \\
  B &=& \alpha_2\beta_1 \\
  A^2 + B^2 & = & C^2 
\end{eqnarray*}
Therefore $(A,B,C)$ can be represented as a pythagorean triple
(\cite{koblitz}, page pp. 1--2)
\begin{eqnarray*}
   A &=& P^2  - Q^2 \\
   B &=& 2PQ \\
   C &=& P^2  + Q^2
\end{eqnarray*}
It is an easy consequence of Theorem \ref{lem:2} that
$(A,B,C)$ is a primitive pythagorean triple, so the representation exists.
The above gets you from $\alpha = \alpha_1/\alpha_2$
and $\beta = \beta_1/\beta_2$ to integers $A,B,C$ and thus to $P$, $Q$.

To go from $P$ and $Q$ to $\alpha$ and $\beta$, use the
above to get $A$ and $B$ and use the following to compute $D = \alpha_2\beta_2$
and then $\alpha_1$ and $\alpha_2$.
\begin{eqnarray*}
  \frac{\alpha_1}{\alpha_2} \frac{\beta_1}{\beta_2} & = & 2N \\
  \frac{A}{\alpha_2\beta_2} \frac{B}{\alpha_2\beta_2} & = & 2N \\
  \frac{A B}{2} & = &  \alpha_2^2 \beta_2^2 N = D^2N \\
  D & = & \sqrt{\frac{AB}{2N}} \\
  \frac{\alpha_1}{\alpha_2} & = & \frac{A}{D} \\
  d & = & \gcd(A, D) \\
  \alpha_1 & = & A/d \\
  \alpha_2 & = & D/d
\end{eqnarray*}
And similarly for $\beta_1, \beta_2$.
The integers $A/d$ and $D/d$ have a ratio of $\alpha_1/\alpha_2$.  But
are they equal to the original $\alpha_1$ and $\alpha_2$?
Yes, because both pairs $(\alpha_1, \alpha_2)$ and $(A/d, D/d)$ are relatively
prime and have the same quotient.

We can now explain the glitch for $N=559$ in the tables  \cite{congruum}
and \cite{numeri}.   For that $N$ they give
$P = 2608225$ and $Q = 4489$, which are both odd.  But the $P$ and $Q$
coming from a primitive pythagorean triple always have opposite parities.

There is an even more compact representation of the
triangle sides than using $P$ and $Q$.

\begin{thm}  \label{lem:square}
  $P = P_0P_1^2$ and $Q = Q_0Q_1^2$ where $P_0Q_0 \mid N$.
\end{thm}
\begin{proof}
\begin{eqnarray}
  2PQ(P^2 - Q^2) &=& AB \nonumber \\
  &=&  2N \alpha_2^2 \beta_2^2 \nonumber \\
PQ(P-Q)(P+Q)  &=&  N \alpha_2^2 \beta_2^2 \label{eq:square}
\end{eqnarray}
If a prime $p$ divides $P$, then it either divides $N$
or $\alpha_2^2 \beta_2^2$.  If the latter, then $p^2$ divides
both sides of (\ref{eq:square}).  But $(P,Q) = 1$ means that
$p \nmid Q$ and so $p^2 \mid P$.  Consequently, for each divisor $p \mid P$,
either $p \mid N$ or $p^2 \mid P$.  Rewrite  (\ref{eq:square})
by dividing out all the factors of $P$, and then repeat the argument
for $Q$.
\end{proof}

As an example of the compression, for $N = 53$, instead of
$P = 1873180325$ and $Q = 1158313156$
you have $P = 53 \cdot 5945^2$ and $Q = 34034^2$.		

\section{Strategy for Generating the Table}

You can get sides $\alpha$ and $\beta$ for a congruent number $N$
by finding a rational point on the elliptic curve $y^2 = x^3 - N^2x$.
See \cite{koblitz}, Proposition 19 on page 47.  So we compute
the sides by using SageMath \cite{sagemath} to get rational points on an elliptic curve.
Details are in the next section.

In SageMath, we create an elliptic curve using
\fnt{e = EllipticCurve([-N*N, 0])}.   Then we use either
\fnt{e.gens} or \fnt{pari(e).ellheegner()} to get rational points.
The first uses Cremona’s mwrank C library, the second uses 
the Heegner point method, described in \cite{cohen}.

We need both methods because Heegner only works for points of rank one, but e.gens
is extremely slow for some $N$.  Fortunately for each $ N < 10000$, one of the
two methods will produce a rational side in a reasonable amount of elapsed computation
time. The use of Heegner points is mentioned in \cite{elkies}
for $N =  1063$.

%

We use two heuristics to try and get sides of minimal height.
For points of rank greater than one, if $g_j$ are generators, 
we get sides for the rational point $g_i$ and $g_i \pm g_j$ for $i < j$, and then pick the
set of sides of minimal height.  For the Heegner point $g$,
we try dividing $g$ by $k = 10, 9, \ldots 2$.  If one of those divisions succeeds,
we use the sides from $g/k$.   Because the group of rational points
for these curves have torsion elements, we 
actually test for divisibility of $g + t$ where $t$ ranges over the torsion subgroup.

The effectiveness of these heuristics was tested by an exhaustive
computation that computed, for each $N$,  the smallest height representation with $Q < P < 25000$.
These never had less height than the representation using our heuristics.

Here are a few numbers related to the heuristics.  For mwrank, there are 447  congruent numbers $N$
where $y^2 = x^3 - N^2x$ has  two generators, and 15 with three generators.  For Heegner points,
see table \ref{tbl:heeg}.  But note that this table only considers divisors $k \leq 10$.

\begin{table}
\begin{tabular}{|r|r|r|} \hline
 $k$ & occurrences & examples of $N$ \\ \hline
  2 & 237 & 205, 221, 438 \\
  3 & 60 & 517, 751, 1103 \\
  4 & 4 & 3805, 6198, 9118, 9143 \\
  5 & 2 & 3093, 6887 \\ \hline
\end{tabular}
\caption{Frequency of Heegner points $P$ that have a $k$-division point $Q$, that is $kQ = P$
  on the curve  $y^2 = x^3 - N^2x$. For each $k < 10$, 
  we list the number of occurrences of
  a congruent number $N$ less than 10000 that have $k$-division points, together with some sample $N$.}
\label{tbl:heeg}
\end{table}

\section{Details on Generating the Table}

If $(x,y)$ is a point on the elliptic curve $y^2 = x^3 - N^2x$, write
\begin{equation} \label{eq:st}
  x = \frac{s}{t} \qquad y = \frac{u}{v}
\end{equation}
Then $(x', y') = 2(x,y)$ is given by 
\begin{eqnarray*}
  x' & = & -2\frac{s}{t} + \left( \frac{3(\frac{s}{t})^2 - N^2}{\frac{2u}{v}} \right)^2 \\
  & = & -\frac{2s}{t} + \left(\frac{3s^2v - t^2vn^2}{2ut^2}\right)^2 \\
  & = & -\frac{2s}{t} + \frac{v^2}{4u^2t^4}\left(3s^2 - t^2n^2 \right)^2 \\
  & = & -\frac{8su^2t^3}{4u^2t^4} + \frac{v^2\left(3s^2 - t^2n^2 \right)^2}{4u^2t^4} \\
  & = & \frac{-8su^2t^3 + v^2\left(3s^2 - t^2n^2 \right)^2}{(2ut^2)^2} \\
  & = & \frac{s'}{t'} = \frac{s'}{\tau^2}
\end{eqnarray*}

Next, use Koblitz, Proposition 19 on page 47, which says that if $(x',y')$ is the sum of a point on
$y^2 = x^3 - N^2x$ with itself, then you get rational sides with area N for sides $\alpha, \beta$ with
\begin{eqnarray*}
  \alpha & = & \sqrt{x'+N} - \sqrt{x'-N} \\
  \beta & = & \sqrt{x'+N} + \sqrt{x'-N} \\
\end{eqnarray*}
so
\begin{eqnarray*}
  \alpha & = & \sqrt{x'+N} - \sqrt{x'-N} = \frac{\sqrt{s' + Nt'} - \sqrt{s' - Nt'}}{\tau} \\
  \beta & = & \sqrt{x'+N} + \sqrt{x'-N} = \frac{\sqrt{s' + Nt'} + \sqrt{s' - Nt'}}{\tau}
\end{eqnarray*}
Or since $\alpha = \alpha_1/\alpha_2$, $\beta = \beta_1/\beta_2$
\begin{eqnarray*}
  \alpha_1 &=& \sqrt{s' + Nt'} - \sqrt{s' - Nt'} \\
  \alpha_2 &=& \tau \\
  \beta_1 &=& \sqrt{s' + Nt'} + \sqrt{s' - Nt'}\\
  \beta_2 &=& d
\end{eqnarray*}

The python code using these formulas is available on GitHub as \url{https://github.com/dgpaloalto/Congruent-Numbers}.


\section{First Rows of the Table}

Here are the first nineteen lines of the table.  The columns $P_0$, $P_1$, $Q_0$, and $Q_1$
are defined in Theorem \ref{lem:square}.

\vspace{1ex}
\begin{scriptsize}
\begin{tabular}{|r|r|r|r|r|r|r|r|r|r|r|r|} \hline
$n$ & $P$ & $Q$ & $\alpha_1$ & $\alpha_2$ & $\beta_1$ & $\beta_2$ & $P_0$ & $P_1$ & $Q_0$ & $Q_1$ & weight \\ \hline 
5 & 5 & 4 & 3 & 2 & 20 & 3 & 5 & 1 & 1 & 2 & 20   \\
6 & 2 & 1 & 3 & 1 & 4 & 1 & 2 & 1 & 1 & 1 & 4   \\
7 & 16 & 9 & 35 & 12 & 24 & 5 & 1 & 4 & 1 & 3 & 35   \\
13 & 325 & 36 & 780 & 323 & 323 & 30 & 13 & 5 & 1 & 6 & 780   \\
14 & 8 & 1 & 8 & 3 & 21 & 2 & 2 & 2 & 1 & 1 & 21   \\
15 & 4 & 1 & 4 & 1 & 15 & 2 & 1 & 2 & 1 & 1 & 15   \\
21 & 4 & 3 & 7 & 2 & 12 & 1 & 1 & 2 & 3 & 1 & 12   \\
22 & 50 & 49 & 33 & 35 & 140 & 3 & 2 & 5 & 1 & 7 & 140   \\
23 & 24336 & 17689 & 80155 & 20748 & 41496 & 3485 & 1 & 156 & 1 & 133 & 80155   \\
29 & 4901 & 4900 & 99 & 910 & 52780 & 99 & 29 & 13 & 1 & 70 & 52780   \\
30 & 3 & 2 & 5 & 1 & 12 & 1 & 3 & 1 & 2 & 1 & 12   \\
31 & 1600 & 81 & 720 & 287 & 8897 & 360 & 1 & 40 & 1 & 9 & 8897  \\
34 & 9 & 8 & 17 & 6 & 24 & 1 & 1 & 3 & 2 & 2 & 24 \\
37 & 777925 & 1764 & 450660 & 777923 & 777923 & 6090 & 37 & 145 & 1 & 42 & 777923 \\
38 & 1250 & 289 & 1700 & 279 & 5301 & 425 & 2 & 25 & 1 & 17 & 5301 \\
39 & 13 & 12 & 5 & 2 & 156 & 5 & 13 & 1 & 3 & 2 & 156 \\
41 & 25 & 16 & 123 & 20 & 40 & 3 & 1 & 5 & 1 & 4 & 123 \\
46 & 72 & 49 & 253 & 42 & 168 & 11 & 2 & 6 & 1 & 7 & 253 \\
47 & 14561856 & 2289169 & 11547216 & 2097655 & 98589785 & 5773608 & 1 & 3816 & 1 & 1513 & 98589785 \\ \hline
\end{tabular}
\end{scriptsize}

\section{Acknowledgements} 
Thanks to Michael Brundage who brought to my attention the use
of Heegner points and \fnt{pari(e).ellheegner}.

\bibliographystyle{abbrv}

\bibliography{congruent.bib} 

\begin{thebibliography}{1}

\bibitem{congruum}
Congruum g : 1 <= g <= 999.
\newblock
  \url{http://www.asahi-net.or.jp/~KC2H-MSM/mathland/math10/matb2000.htm}.
\newblock Accessed: 2021-05-01.

\bibitem{numeri}
Numeri congruenti minori di 1000.
\newblock \url{http://bitman.name/math/table/29}.
\newblock Accessed: 2021-03-21.

\bibitem{cohen}
H.~Cohen.
\newblock {\em Number Theory Volume I: Tools and Diophantine Equations}.
\newblock Springer, New York, 2007.

\bibitem{elkies}
N.~D. Elkies.
\newblock Heegner point computations.
\newblock In {\em International Algorithmic Number Theory Symposium}, pages
  122--133. Springer, 1994.

\bibitem{koblitz}
N.~Koblitz.
\newblock {\em Introduction to Elliptic Curves and Modular Forms}.
\newblock Springer-Verlag, New York, 1993.

\bibitem{sagemath}
{The Sage Developers}.
\newblock {\em {S}ageMath, the {S}age {M}athematics {S}oftware {S}ystem
  ({V}ersion 9.2)}, 2020.
\newblock {\tt https://www.sagemath.org}.

\end{thebibliography}

\end{document}